\documentclass[11pt]{amsart}

\usepackage{amsfonts, amsmath, amssymb, color}
\usepackage{amsthm}
\usepackage{booktabs}
\usepackage{float}
\usepackage{hhline}
\usepackage{lipsum}
\usepackage{lscape}
\usepackage{subfig}
\usepackage{textcomp}
\usepackage{tikz}
\usetikzlibrary{decorations.pathmorphing,shapes,arrows,positioning}
\usepackage{xcolor}
\usepackage{young}
\usepackage{youngtab}
\usepackage{ytableau}

\theoremstyle{plain}
\newtheorem{thm}{Theorem}[section]
\newtheorem{lem}[thm]{Lemma}
\newtheorem{prop}[thm]{Proposition}
\newtheorem{cor}[thm]{Corollary}
\newtheorem{rem}[thm]{Remark}

\newtheorem{defn}[thm]{Definition}

\numberwithin{equation}{section} \errorcontextlines=0

\newcommand{\la}{\lambda}

\begin{document}
\title{An iterative formula for the Kostka-Foulkes polynomials}
\author{Timothee W. Bryan}
\address{Department of Mathematics, Colgate University,
Hamilton, NY 13346, USA}
\email{tbryan@colgate.edu}
\author{Naihuan Jing$^*$}
\address{Department of Mathematics, North Carolina State University, Raleigh, NC 27695, USA}
\email{jing@ncsu.edu}
\thanks{$*$Corresponding author: jing@ncsu.edu}

\maketitle

\begin{abstract}
An iterative formula for the Kostka-Foulkes polynomials is given using the vertex operator realization of the Hall-Littlewood polynomials.
The operational formula can handle large Kostka-Foulkes polynomials, and a stability property for
the Kostka-Foulkes polynomials is shown. We also use our algorithm to give a formula of $K_{\lambda\mu}(t)$
for $\mu$ being hook-shaped.
\end{abstract}

\section{Introduction}

The Kostka-Foulkes polynomials $K_{\lambda\mu}(t)$ associated to a pair of partitions $\lambda, \mu$ have figured prominently in
representation theory and algebraic combinatorics (cf. \cite{DLT, Ki2}). They have arisen in the context of character theory of the general
linear group over the finite field \cite{G} in terms of the Hall-Littlewood polynomials \cite{Lit}.
Foulkes \cite{F} conjectured that $K_{\lambda\mu}(t)$ are
positive integral polynomials, which was proved by Lascoux-Sch\"utzenberger \cite{LS} using plactic monoids.
Lusztig later proved the Foulkes conjecture using geometric method \cite{L} by identifying
$q^{n(\mu)-n(\lambda)}K_{\lambda\mu}(q^{-1})$ as the affine Kazhdan-Lusztig polynomials of type $A$.
It was also proved using crystal graphs \cite{NY}. More generally it is a special case of the positivity
of the $(q, t)$-Kostka-Foulkes polynomials \cite{H}.
The Kostka-Foulkes
polynomials also calculate some fusion products for flag manifolds \cite{Ke}.
In combinatorics, $K_{\lambda\mu}(1)$ counts the number of tableaux of shape $\lambda$ with weight $\mu$, which is
the dimension of the irreducible representation of the symmetric group and also determines the characters
of the irreducible representations of the special linear Lie algebras.

Theoretically the Kostka-Foulkes polynomials can be expressed in several ways.
The first combinatorial formula is Lascoux-Sch\"utzen\-ber\-ger's charge statistic which shows that
the Kostka-Foulkes polynomials are positive integral polynomials. In terms of Kostant's partition function, Kostka-Foulkes polynomials are
given inductively by symmetrizing over the
symmetric group \cite[Ch. 3]{M}. Kirillov and Reshetikhin have derived a combinatorial formula
\cite{KR} via the Bethe Ansatz. Other combinatorial formulas \cite{GT, Ki1, HKO, DS} have been derived
since the introduction of the Macdonald symmetric functions. The Kostka-Foulkes polynomials can also be expressed in terms of crystal graphs \cite{NY}, which also shows the positivity.
The Haglund-Haiman-Loehr statistic \cite{HHL} for the general Macdonald polynomial obviously implies the positivity in one-parameter case.
Recently a positive combinatorial formula is known for the symplectic Kostka-Foulkes polynomial \cite{DGT}.

On the computational side, Morris has given an implicit iterative procedure to compute the Kostka-Foulkes polynomials \cite{Morris}. The
algorithm compiles
previous Hall-Littlewood polynomials to compare with relevant Schur functions to find the next one
(cf. \cite{DLT}), which is sometimes complicated to pin down the solution.
This poses a question on how to find a {\it direct and efficient} formula for the Kostka-Foulkes polynomials.

The aim of this paper is to give an operational and iterative formula for the Kostka-Foulkes polynomials, using
algebraic properties of the vertex operators of the Hall-Littlewood functions given by the second-named author \cite{Jing1}.
Our explicit formula facilitates the computation and may help reveal intrinsic property of $K_{\lambda\mu}(t)$, for instance, a stability property for the Kostka-Foulkes polynomial is derived (see Proposition \ref{p:stability}).

Although the formula was established
via the vertex operator techniques, the final computation can be formulated independently. One of the motivations is that our formula might shed some light on
an iterative formula for the two-parameter Kostka-Foulkes
polynomials.

In \cite{Z} Zabrocki has studied
certain hook-shaped operator to compute the Hall-Littlewood functions. It would be interesting to study an analogous
column operation similar to our formula for the Kostka-Foulkes polynomials.

The second-named author would like to thank Adriano Garsia for interesting discussions on
using vertex operators to compute the Kostka-Foulkes polynomials back in 1989.
Our current approach to the problem
is based upon dual vertex operators developed in \cite{Jing1, Jing2}.

\section{Vertex operator realization of Hall-Littlewood polynomials}

The original Hall-Littlewood vertex operators defined
in \cite{Jing1} used the infinite dimensional Heisenberg algebra and
its irreducible representation parameterized by $t$. These vertex operators were subsequently utilized in \cite{Jing2} to
study more realizations of symmetric functions.  To
be consistent with combinatorial convention, we reformulate the vertex operators
for the Hall-Littlewood symmetric functions as follows. For plethestic formulation, see \cite{SZ}.
Moreover, we will also recall a crucial vertex operator to realize the Schur functions in the same ring \cite{Jing1} as well as
its dual vertex operator \cite{Jing2}.

Let $V$ be the space of symmetric functions over $\mathbb Q(t)$.
There are several well-known bases such as ele\-mentary symmetric functions, power-sum symmetric functions, monomial symmetric functions,
and homogeneous symmetric functions.
The power-sum symmetric functions are perhaps the most useful ones on the space, as they are group-like elements
under the usual Hopf algebra structure.

A partition $\lambda$ of $n$, denoted $\la\vdash n$, is a non-increasing sequence of nonnegative integers: $\lambda_1\geqslant \lambda_2\geqslant \ldots \geqslant\lambda_l> 0$ such that $\sum_i\la_i=n$,
where $l=l(\lambda)$ is the length of the partition $\lambda$. The dominance order $\geqslant$ is defined by
$\lambda\geqslant\mu$ if $\lambda_1+\cdots+\lambda_i\geqslant\mu_1+\cdots+\mu_i$ for all $i$. For each partition $\lambda$, we define
\begin{align}
n(\lambda)=\sum_{i=1}^{l(\lambda)}(i-1)\la_i.
\end{align}

For partition $\la=(\la_1, \ldots, \la_l)$, we define
\begin{align}
\la^{[i]}&=(\la_{i+1}, \cdots, \la_l), \qquad i=0, 1, \ldots, l\\ \label{e:par}
\lambda^{(i)}&=(\lambda_{1}+1,\ldots ,\lambda_{i-1}+1,\lambda_{i+1},\ldots ,\lambda_{l}).
\end{align}
So $\la^{[0]}=\la$ and $\la^{[l]}=\emptyset$. It is readily seen that $n(\la^{[i]})-n(\la^{[i+1]})=|\la^{[i+1]}|$.

The partition $\la$ is usually visualized by its Young diagram of aligning $l$ rows of boxes to the left where the $i$th row consists
of $\la_i$ boxes. The dual partition $\la'=(\la_1', \ldots, \la_{\la_1}')$ is the partition associated with the reflection of
the Young diagram of $\la$ along the diagonal.

Recall that the power sum symmetric functions $p_{\lambda}$ form a $\mathbb Q$-basis of the ring of the symmetric functions, where $p_{\lambda}=p_{\lambda_1}p_{\lambda_2}\cdots$ with
\begin{equation}
p_n=x_1^n+x_2^n+\cdots.
\end{equation}
Therefore the ring of symmetric functions $V=
\mathbb Q(t)[p_1, p_2, \ldots]$, the ring of polynomials in the $p_n$.
Using the degree gradation, $V$ becomes a graded ring
\begin{align}
V=\bigoplus_{n=0}^{\infty} V_n.
\end{align}
A linear operator $f$ is of degree $n$ if $f(V_m)\subset V_{m+n}$.

Let $\langle\ , \ \rangle$ be the Hall-Littlewood bilinear form on $V$ over $\mathbb Q(t)$ defined by
\begin{align}
\langle p_{\lambda}, p_{\mu}\rangle=\delta_{\lambda\mu}\prod_{i\geqslant 1}\frac{i^{m_i(\lambda)}m_i(\lambda)!}{1-t^{\lambda_i}},
\end{align}
where $m_i(\lambda)$ is the multiplicity of $i$ in the partition $\lambda$. The multiplication operator
$p_n: V\longrightarrow V$ is of degree $n$. 
Clearly, the dual operator is the differential operator $p_n^* =\frac{n}{(1-t^n)}\frac{\partial}{\partial p_n}$ of degree $-n$. Note that * is $\mathbb Q(t)$-linear and an anti-involution.
\begin{defn}
The \textit{vertex operators} $S(z)$ and $H(z)$ and their dual operators $S^*(z)$ and $H^*(z)$ on the space $V$ are defined as the following
maps $V\longrightarrow V[[z, z^{-1}]]$ given by
\begin{align*}
S(z)&=\mbox{exp}\left( \sum\limits_{n\geqslant 1} \dfrac{1}{n}p_{n}z^{n} \right) \mbox{exp} \left( -\sum \limits_{n\geqslant 1} \frac{\partial}{\partial p_{n}}z^{-n} \right)=\sum_{n\in\mathbb Z}S_nz^{n},\\
S^*(z)&=\mbox{exp}\left(-\sum\limits_{n\geqslant 1} \dfrac{1-t^n}{n}p_{n}z^{n}\right) \mbox{exp} \left(\sum \limits_{n\geqslant 1} \frac{z^{-n}}{1-t^n}\frac{\partial}{\partial p_{n}}\right)=\sum_{n\in\mathbb Z}S^*_nz^{-n}.
\end{align*}
\begin{align*} 
H(z)&=\mbox{exp} \left( \sum\limits_{n\geqslant 1} \dfrac{1-t^{n}}{n}p_nz^{n} \right) \mbox{exp} \left( -\sum \limits_{n\geqslant 1} \frac{\partial}{\partial p_n}z^{-n} \right)=\sum_{n\in\mathbb Z}H_nz^{n},\\
H^*(z)&=\mbox{exp} \left(-\sum\limits_{n\geqslant 1} \dfrac{1-t^{n}}{n}p_nz^{n} \right) \mbox{exp} \left(\sum \limits_{n\geqslant 1} \frac{\partial}{\partial p_n}z^{-n} \right)=\sum_{n\in\mathbb Z}H^*_nz^{-n}.
\end{align*}
\end{defn}
\begin{rem}
\textnormal{
The dual operators $S^*(z)$ and $H^*(z)$ \cite{Jing1, Jing2} will be the key for our later discussion.
The indexing is different from that of \cite{Jing1, Jing2}. Technically $S(z)^*=S^*(z^{-1})$ in the current notation.
}
\end{rem}

We collect some useful relations from \cite{Jing1} for explicit computation.

\begin{lem} We have the following relations:
\begin{align*}
H_{n}H_{n+1}&=tH_{n+1}H_{n}, \qquad H_{n}^{*}H_{n-1}^{*}=tH_{n-1}^{*}H_{n}^{*}, \quad n\in\mathbb Z;\\
H_{-n}. 1&=\delta_{n, 0}, \qquad\qquad H_{n}^{*}. 1=\delta_{n, 0}, \quad n\geqslant 0
\end{align*}
where $\delta_{n,m}$ is the Kronecker delta and $1$ is the vacuum vector in $V$.
\end{lem}
Observe that for $n\geqslant 0$, both $H_{-n}$ and $H_{n}^*$ are annihilation operators of degree $-n$.

We introduce the symmetric function $q_n=q_n(x; t)$ by the generating series
\begin{align}\label{e:genhomog}
\sum_{n\geqslant 0} q_nz^n=exp(\sum_{n=1}^{\infty}\frac{1-t^n}np_n z^n).
\end{align}
Clearly, as a polynomial in the $p_k$, the polynomial $q_n$ is the Hall-Littlewood symmetric function $Q_{(n)}$ associated to
the one-row partition $(n)$. For any partition $\lambda=(\lambda_1, \lambda_2, \cdots )$, we define
$$q_{\lambda}=q_{\lambda_1}q_{\lambda_2}\cdots$$
then the set $\{q_{\lambda}\}$ forms a basis of $V$,
which are usually called the generalized homogenous polynomials.

For convenience, we also denote $s_n=q_n(0)$, the homogeneous symmetric function or
the Schur function associated with the row partition $(n)$. The generating series of
the $s_n$ is
\begin{align}\label{e:homog}
s(z)=\sum_{n\geqslant 0} s_nz^n=exp(\sum_{n=1}^{\infty}\frac 1np_n z^n).
\end{align}

In this paper, we adopt the combinatorial $t$-integer $[n]$, $n\in\mathbb N$ defined by
\begin{align}
[n]=\frac{1-t^n}{1-t}=1+t+\cdots+t^{n-1}.
\end{align}

Therefore $[n]!=[n][n-1]\cdots [1]$. For simplicity we define $[0]=1$. We also use the Gauss $t$-binomial symbol
\begin{align}
\left[\begin{matrix}n\\ k\end{matrix}\right]=\frac{[n]!}{[k]![n-k]!}
\end{align}
which can also be defined inductively by:
\begin{align}\label{e:binom2}
\left[\begin{matrix}n+1\\ k\end{matrix}\right]=t^k\left[\begin{matrix}n\\ k\end{matrix}\right]+\left[\begin{matrix}n\\ k-1\end{matrix}\right]
\end{align}

\begin{prop} \cite{Jing1} \label{P:Jing1} \label{P:HL}
(1) Given a partition $\lambda=(\lambda_{1},\lambda_{2},\ldots ,\lambda_{l})$, the product
  $H_{\lambda_{1}}H_{\lambda_{2}}\cdots H_{\lambda_{l}}. 1$ can be expressed as
$$H_{\lambda_{1}}H_{\lambda_{2}}\cdots H_{\lambda_{l}}. 1
= \prod\limits_{i<j} \dfrac{1-R_{ij}}{1-tR_{ij}}q_{\lambda_{1}}q_{\lambda_{2}}\cdots q_{\lambda_{l}}
$$
where $R_{ij}$ is the raising operator given by $$R_{ij}q_{(\mu_{1},\mu_{2},\ldots ,\mu_{l})}=q_{(\mu_{1},\mu_{2},\ldots ,\mu_{i}+1,\ldots ,\mu_{j}-1,\ldots , \mu_{l})}$$
which ensures that the expression has finitely many terms. Moreover, the products
 $H_{\lambda}.1=H_{\lambda_1}H_{\lambda_2}\cdots H_{\lambda_l}.1$ are orthogonal such that
\begin{align}\label{e:orth}
\langle H_{\lambda}.1, H_{\mu}.1\rangle=\delta_{\lambda\mu}b_{\lambda}(t),
\end{align}
where $b_{\lambda}(t)=(1-t)^{l(\lambda)}\prod_{i\geqslant 1}[m_i(\lambda)]!$.

(2) Similarly for composition $\mu=(\mu_1, \cdots, \mu_k)$, the elements  $S_{\mu_{1}}S_{\mu_{2}}\cdots S_{\mu_{k}}. 1$ can be expressed as
$$S_{\mu_{1}}S_{\mu_{2}}\cdots S_{\mu_{k}}. 1 = \prod\limits_{i<j}(1-R_{ij})
s_{\mu_{1}}s_{\mu_{2}}\cdots s_{\mu_{k}}=s_{\mu},
$$
which is the Schur function associated to the composition $\mu$. In general, $s_{\mu}=0$ or $\pm s_{\lambda}$ for a partition $\lambda$ such that
$\lambda\in \mathfrak S_{l}(\mu+\delta)-\delta$. Here $\delta=(l-1, l-2, \cdots, 1, 0)$, where $l=l(\lambda)$.
\end{prop}

The vertex operator $H_{\lambda}.1$ is the Hall-Littlewood polynomial
$Q_{\lambda}(t)$ defined in \cite{M}. Similarly
$S_{\lambda}.1$ is the Schur polynomial $s_{\lambda}$, therefore we will sometimes write
$S_{\lambda}=S_{\lambda_1}\cdots S_{\lambda_l}.1$ (similarly for $H_{\lambda}=H_{\lambda}.1$) if there is no confusion.

\begin{defn}
The \textit{Kostka-Foulkes polynomials} $K_{\lambda \mu}(t)$ are defined for all partitions $\mu, \lambda$ by
$$s_{\lambda}=\sum\limits_{\mu} \dfrac{1}{b_{\mu}(t)}K_{\lambda \mu}(t)Q_{\mu}(t)$$
where $b_{\lambda}(t)=(1-t)^{l(\la)}\prod_{i\geqslant 1}[m_i(\la)]!$ was defined in \eqref{e:orth}.
\end{defn}
It is known that $K_{\la\mu}(t)=0$ unless $\la\geqslant \mu$. Using Proposition \ref{P:HL}, one immediately gets the following:
\begin{equation}\label{e:KFpoly}
K_{\lambda \mu}(t)= \langle S_{\lambda}.1,H_{\mu}.1 \rangle=\langle H^*_{\mu_1}S_{\lambda}.1,H_{\mu^{[1]}}.1 \rangle.
\end{equation}
Therefore if $H^*_{\mu_1}S_{\lambda}.1$ is straightened out as a linear combination of $S_{\tau}.1$ with $|\tau|=|\lambda|-\mu_1$,
then one can compute $K_{\lambda\mu}(t)$ iteratively.

\begin{prop}
\label{Up Down}
The commutation relations between the vertex operators realizing Hall-Littlewood and Schur symmetric functions are:
\begin{align}\label{e:com1}
S_ms_n&=s_nS_m-s_{n-1}S_{m+1},\\ \label{e:com2}
H_m^*S_n&=t^{-1}S_nH_m^*+t^{-1}H_{m-1}^*S_{n-1}+(t^{n-m}-t^{n-m-1})s_{n-m}.
\end{align}
\end{prop}
\begin{proof} The first relation follows from the relation
$$S(z)s(w)=(1-\dfrac{w}{z})s(w)S(z)$$
where $s(z)$ is the generating series of the Schur function \eqref{e:homog}.

The second relation follows from the commutation relation
\begin{align}
H^*(z)S(w)\frac{w-tz}z+S(w)H^*(z)=(1-t)s(tz)\delta(\frac wz)
\end{align}
which can be checked by the same method in \cite{Jing1} and
$\delta(z)=\sum_{n\in\mathbb Z}z^n$ is the delta function.
\end{proof}


We remark that in applying Proposition \ref{Up Down} one often utilizes the simple fact that
$$ H^*_nS_{\lambda_{1}}S_{\lambda_{2}}\cdots S_{\lambda_{l}}.1=0$$
whenever $n>|\lambda|$, which helps terminate the iteration.

We are now ready to give our main result, an explicit formula for the Kostka-Foulkes polynomials utilizing the vertex operator for the Hall-Little\-wood polynomials. In the theorem and subsequent proof, we will make repeated use of partitions $\lambda^{(i)}$ defined
in \eqref{e:par}.

\begin{thm}\label{t:iterative}
\label{Kostka Formula}
For partition $\lambda\vdash n$ with $\lambda=(\lambda_{1},\lambda_{2},\ldots ,\lambda_{l})$ and natural number $k$,
\begin{align}\label{e:iterative}
H_{k}^{*}S_{\lambda}=\sum\limits_{i=1}^l (-1)^{i-1}t^{\lambda_{i}-k-i+1}s_{\lambda_{i}-k-i+1}S_{\lambda^{(i)}},
\end{align}
where $s_{n}$ is the multiplication operator by the Schur polynomial $s_{n}$.
\end{thm}

\begin{proof} We argue by induction on $|\lambda |+k$, where $\lambda$ is a partition and $k$ is the degree of
the dual vertex operator for the Hall-Littlewood symmetric function.
First of all, the initial step is clear.
Thus the inductive hypothesis implies that
\begin{align*}
&H_{k-1}^{*}S_{\lambda_{1}-1}S_{\lambda_{2}}\cdots S_{\lambda_{l}}\\
&=t^{\lambda_{1}-k}s_{\lambda_{1}-k}S_{\lambda_{2}}\cdots S_{\lambda_{l}}-t^{\lambda_{2}-k}s_{\lambda_{2}-k}S_{\lambda_{1}}S_{\lambda_{3}}\cdots S_{\lambda_{l}}-\cdots \\
&H_{k}^{*}S_{\lambda_{2}}\cdots S_{\lambda_{l}}\\
&=t^{\lambda_{2}-k}s_{\lambda_{2}-k}S_{\lambda_{3}}\cdots  S_{\lambda_{l}}-t^{\lambda_{3}-k-1}s_{\lambda_{3}-k-1}S_{\lambda_{2}+1}S_{\lambda_{4}}\cdots  S_{\lambda_{l}}-\cdots.
\end{align*}

Now we have
\begin{align*}
H_{k}^{*}S_{\lambda}
&=t^{-1}S_{\lambda_{1}}H_{k}^{*}S_{\lambda_{2}}\cdots S_{\lambda_{l}}+t^{-1}H_{k-1}^{*}S_{\lambda_{1}-1}S_{\lambda_{2}}\cdots S_{\lambda_{l}}\\
&\quad+(t^{\lambda_{1}-k}-t^{\lambda_{1}-k-1})s_{\lambda_{1}-k}S_{\lambda_{2}}\cdots S_{\lambda_{l}}\\
&= t^{-1}S_{\lambda_{1}}(t^{\lambda_{2}-k}s_{\lambda_{2}-k}S_{\lambda_{3}}\cdots S_{\lambda_{l}}-t^{\lambda_{3}-k-1}s_{\lambda_{3}-k-1}S_{\lambda_{2}+1}S_{\lambda_{4}}\cdots S_{\lambda_{l}}\\
&\quad+t^{\lambda_{4}-k-2}s_{\lambda_{4}-k-2}S_{\lambda_{2}+1}S_{\lambda_{3}+1}S_{\lambda_{5}}\cdots S_{\lambda_{l}}-\cdots)\\
&\quad+t^{-1}(t^{\lambda_{1}-k}s_{\lambda_{1}-k}S_{\lambda_{2}}\cdots S_{\lambda_{l}}-t^{\lambda_{2}-k}s_{\lambda_{2}-k}S_{\lambda_{1}}S_{\lambda_{3}}\cdots S_{\lambda_{l}}\\
&\quad+t^{\lambda_{3}-k-1}s_{\lambda_{3}-k-1}S_{\lambda_{1}}S_{\lambda_{2}+1}S_{\lambda_{4}}\cdots S_{\lambda_{l}}-\cdots \\
&\quad+(t^{\lambda_{1}-k}-t^{\lambda_{1}-k-1})s_{\lambda_{1}-k}S_{\lambda_{2}}\cdots S_{\lambda_{l}}
\end{align*}
Using the commutation relation \eqref{e:com1} in Proposition \ref{Up Down}, the first parenthesis can be put into
\begin{align*}
&\ (t^{\lambda_{2}-k-1}s_{\lambda_{2}-k}S_{\lambda_{1}}S_{\lambda_{3}}\cdots S_{\lambda_{l}}-t^{\lambda_{3}-k-2}s_{\lambda_{3}-k-1}S_{\lambda_{1}}S_{\lambda_{2}+1}S_{\lambda_{4}}\cdots S_{\lambda_{l}}\\
&+t^{\lambda_{4}-k-3}s_{\lambda_{4}-k-2}S_{\lambda_{1}}S_{\lambda_{2}+1}S_{\lambda_{3}+1}S_{\lambda_{5}}\cdots S_{\lambda_{l}}-\cdots )\\
&-(t^{\lambda_{2}-k-1}s_{\lambda_{2}-k-1}S_{\lambda_{1}+1}S_{\lambda_{3}}\cdots S_{\lambda_{l}}-t^{\lambda_{3}-k-2}s_{\lambda_{3}-k-2}S_{\lambda_{1}+1}S_{\lambda_{2}+1}S_{\lambda_{4}}\cdots S_{\lambda_{l}}\\
&+t^{\lambda_{4}-k-3}s_{\lambda_{4}-k-3}S_{\lambda_{1}+1}S_{\lambda_{2}+1}S_{\lambda_{3}+1}S_{\lambda_{5}}\cdots S_{\lambda_{l}}-\cdots )
\end{align*}
Combining with the other terms, they are exactly the following sum
\begin{align*}
\sum\limits_{i=1}^{l} (-1)^{i-1}t^{\lambda_{i}-k-i+1}s_{\lambda_{i}-k-i+1}S_{\lambda^{(i)}}.
\end{align*}
\end{proof}

Using the Pieri rule, we get the following direct and
iterative formula for the Kostka-Foulkes polynomials.
\begin{cor}\label{c:alg}
For partition $\lambda, \mu\vdash n$ with $\lambda\geqslant\mu$, one has that
\begin{align}\label{e:iterative2}
K_{\la\mu}(t)=\sum\limits_{i=1}^{l(\lambda)} (-1)^{i-1}t^{\lambda_{i}-\mu_1-i+1}\sum_{\tau^i}K_{\tau^i\mu^{[1]}}(t),
\end{align}
where $\mu^{[1]}=(\mu_2, \mu_3, \ldots)$ and $\tau^i$ runs through the partitions such that $\tau^i/\la^{(i)}$ are horizontal $(\la_i-\mu_1-i+1)$-strips.
\end{cor}

\begin{rem}
\textnormal{
Morris \cite{Morris} has given an implicit iterative algorithm to compute $K_{\la\mu}(t)$ by
listing all Schur functions corresponding to previous partitions in dominance order and then
extract the relevant ones (see \cite[Sect. 4]{DLT} for a clear exposition),
while ours is a direct iterative one based on the upper
partition $\lambda$.
}
\end{rem}

\ In the following remark we highlight several interesting features of the Kostka-Foulkes polynomial formula in Theorem \ref{Kostka Formula}.

\begin{rem} \textnormal{For partitions $\lambda,\mu\vdash n$ with $\lambda=(\lambda_{1},\lambda_{2},\ldots ,\lambda_{l})$,}
\hspace{1.15in} \begin{itemize}
\item \textnormal{ The formula has at most $l$ nonzero summands, although in many cases there are fewer by
$s_{-m}=0$ for a positive integer $m$.}
\item \textnormal{ The formula depends only upon $\mu_{1}$ from the content partition.}
\item \textnormal{ The exponent $\lambda_{i}-k-i+1$ of $t$ is evocative of the exponents of the determinant definition of the Schur symmetric functions.}
\item \textnormal{ Despite the alternating signs in the formula, positivity is assured by \cite{LS}. It would be interesting how to
show the positivity directly from Corollary \ref{c:alg}.}
\end{itemize}
\end{rem}


%

The Kostka number $K_{\lambda\mu}$ satisfies the stability property \cite[Ex.1.6.4]{M}, which
implies that if $\mu_1\geqslant\lambda_2$ then $K_{\lambda+(r),\mu+(r)}=K_{\lambda, \mu}$ for all $r\geqslant 1$.
Here $\lambda+(r)=(\lambda_1+r, \lambda_2, \cdots)$. We have the following $t$-analog:

\begin{prop} \label{p:stability} Suppose $\mu_1\geqslant\lambda_2$, then for all $r\geqslant 1$ we have the stability
$K_{\lambda+(r),\mu+(r)}(t)=K_{\lambda, \mu}(t)$.
\end{prop}
\begin{proof} This follows immediately from Theorem \ref{t:iterative}:
\begin{align*}
K_{\lambda+(r),\mu+(r)}(t)=t^{\lambda_1-\mu_1}\langle s_{\lambda_i-\mu_1}S_{\lambda^{(1)}}, H_{\mu^{[1]}}\rangle=
&K_{\lambda,\mu}(t)
\end{align*}
\end{proof}

The following simple facts are
easy consequences of Theorem \ref{t:iterative}. The first two are well-known (see \cite{M}).

\begin{prop} For $\mu\vdash n$, we have that
\begin{align}\label{e:1row}
&K_{\mu, (n)}=\delta_{\mu, (n)}, \\
&K_{(n), \mu}=t^{n(\mu)},\\
&K_{(r, s), {\mu}}=t^{r-\mu_1}(K_{(r-\mu_1+s), \tilde{\mu}}+K_{(r-\mu_1+s-1, 1), \tilde{\mu}}+\cdots+K_{(r-\mu_1, s), \tilde{\mu}})\\ \notag
&-t^{s-\mu_1-1}(K_{(r-\mu_1+s), \tilde{\mu}}+K_{(r-\mu_1+s-1, 1), \tilde{\mu}}+\cdots+K_{(r+1, s-\mu_1-1), \tilde{\mu}})
\end{align}
where $\tilde{\mu}=\mu^{[1]}$.
\end{prop}

Morris has given a table of $K_{\lambda\mu}$ for $|\lambda|\leqslant 4$
and derived a formula for $K_{(21^{n-2}), \mu}$ in $\mathbb Z[t]$ \cite{Morris}. A general compact formula
is known for $\lambda$ being hook-shaped \cite[Lemma 7.12]{Ki2}:

\begin{align}
K_{(n-k, 1^k), \mu}&=t^{n(\mu)+\frac{k(k+1-2l)}2}\left[\begin{matrix}l-1\\ k\end{matrix}\right]
\end{align}
where $n=|\mu|$, $l=l(\mu)$ and $(n-k, 1^k)\geqslant\mu$. The formula can also be easily proved by induction on $|\mu|$ using Corollary \ref{c:alg}.


\medskip

We can give another formula for the Kostka-Foulkes polynomials of hook-shaped partitions (lower partition $\mu$).
Recall the formula \cite[p. 243]{M}:
\begin{align}\label{e:column}
K_{\la, (1^n)}=t^{n(\lambda')}\frac{[n]!}{\prod_{x\in\lambda}[h(x)]}
\end{align}
where $h(x)=\la_i+\la_j'-i-j+1$ is the hook length of $x=(i, j)$.

\begin{prop} For all $\lambda\geqslant (n-k, 1^k)$
\begin{equation}
K_{\lambda, (n-k, 1^k)}=\sum_{i=1}^{l(\lambda)}(-1)^it^{\lambda_i-n+k-i+1}\sum_{\tau^j}\frac{t^{n((\tau^j)')}[k]!}{\prod_{x\in\tau^j}[h(x)]},
\end{equation}
where $\tau^j$ runs through each partition such that $\tau^j/\lambda^{(i)}$ is a horizontal $(\lambda_i-n+k-i+1)$-strip.
\end{prop}
\begin{proof} It follows from Corollary \ref{c:alg} that  $K_{\lambda, (n-k, 1^k)}$ is written as an alternative sum
in terms of $K_{\tau^j, 1^{k}}$, where $\tau^j$ runs through each partition such that $\tau^j/\lambda^{(i)}$ is a horizontal $(\lambda_i-n+k-i+1)$-strip, the formula is then proved by using \eqref{e:column}.
\end{proof}

\vskip30pt \centerline{\bf Acknowledgments}
The project is partially supported by
Simons Foundation grant Nos. 198129 and 523868 and NSFC grant Nos. 11271138 and 11531004.
\bigskip

\bibliographystyle{plain}

\begin{thebibliography}{6}



\bibitem{DLT} J.~D\'esarm\'enien, B.~Leclerc and J.-Y.~Thibon, {\em Hall-Littlewood functions and Kostka-Foulkes polynomials in representation theory},
S\'em. Lothar. 
Combin. 32 (1994), B32c, 38 pp.

\bibitem{DS} L. Deka, A. Schilling, {\em New fermionic formula for unrestricted Kostka polynomials}.
J. Combin. Theory Ser. A 113 (2006), no. 7, 1435-1461. 

\bibitem{DGT} M. Dolega, T. Gerber, J. Torres, {\em A positive combinatorial formula for symplectic Kostka-Foulkes polynomials}, J. Algebra
560 (2020), 1253-1296.

\bibitem{F} H.~O.~Foulkes, {\em A survey of some combinatorial aspects of symmetric functions}, (Actes Colloq., Univ. Ren\'e-Descartes, Paris, 1972), pp. 79-92. Gauthier-Villars, Paris, 1974.

\bibitem{GT} A. Garsia and G. Tesler, {\em Plethystic formulas for Macdonald $q, t$-Kostka coefficients}, Adv. Math. 123 (1996), 144-222.

\bibitem{G} J.~A. Green, {\em The character of the finite general linear group}, Trans. Amer. Math. Soc. 80 (1955), 402-477.

\bibitem{H} M. Haiman, {\em Hilbert schemes, polygraphs, and the Macdonald positivity conjecture},
J. Amer. Math. Soc. 14 (2001) 941-1006.

\bibitem{HKO} G. Hatayama, A.~N.~Kirillov, A. Kuniba, M.~Okado, T.~Takagi, Y.~Yamada, Yasuhiko,
{\em Character formulae of $sl_n$-modules and inhomogeneous paths}, in: Contemp. Math. 248, 1999, pp. 243-291.

\bibitem{HHL} J. Haglund, M. Haiman, N. Loehr, {\em A combinatorial formula for Macdonald polynomials}, J. Amer. Math. Soc.
18 (2005), 735-761.


\bibitem{Jing1} N. Jing, {\em Vertex operators and Hall-Littlewood symmetric functions}, Adv. Math. 87 (1991), 226-248.

\bibitem{Jing2} N. Jing, {\em Vertex operators and generalized symmetric functions}, In: Quantum Topology,  pp. 111-128, River Edge, NJ, 1994. World Scientific, Singapore.

\bibitem{Ke} R. Kedem, {\em Fusion products, cohomology of $\mathrm{GL}_N$ flag manifolds, and Kostka polynomials},
Int. Math. Res. Not. (2004), 1273-1298.

\bibitem{Ki1} A.~N.~Kirillov, {\em On the Kostka-Green-Foulkes polynomials and Clebsch-Gordon numbers},
J. Geom. Phys. 5 (1988), 365-389. 

\bibitem{Ki2} A.~N.~Kirillov, {\em Ubiquity of Kirillov polynomials}, in: Physics and combinatorics 1999 (Nagoya), pp. 85-200, World Sci. Publ., River Edge, NJ, 2001.  

\bibitem{KR} A.~N.~Kirillov and N.~Yu~Reshetikhin, {\em The Bethe Ansatz and combinatorics of Young tableaux}, J. Soviet Math. 41 (1988), 925-955.

\bibitem{LS} A. Lascoux and M. P. Sch\"utzenberger, {\em Sur une conjecture de H.O. Foulkes}, C.~R. Acad. Sci. 268A (1978), 323-324.

\bibitem{Lit} D. E. Littlewood, {\em On certain symmetric functions}, Proc. London Math. Soc. 11(3) (1961), 485-498.

\bibitem{L} G. Lusztig, {\em Green polynomials and singularities of unipotent classes}, Adv. Math. 42 (1981), 169-178.

\bibitem{M} I. G. Macdonald, {\em Symmetric functions and Hall polynomials}, 2nd ed., Clarendon Press, Oxford, 1995.

\bibitem{Morris} A. O. Morris, {\em The characters of the group $\mathrm{GL}(n,q)$}, Math. Zeit. 81 (1963), 112-123.

\bibitem{NY} A. Nakayashiki, Y. Yamada,
{\em Kostka polynomials and energy functions in solvable lattice models},
Selecta Math. (N.S.) 3 (1997), 547-599.   

\bibitem{SZ} M. Shimozono, M. Zabrocki, {\em Hall-Littlewood vertex operators and generalized Kostka polynomials}, Adv. Math. 158 (2001), 66-85.

\bibitem{Z} M. Zabrocki, {\em On the action of the Hall-Littlewood vertex operator}, Ph.D thesis, Univ. of California, San Diego, 1998.
\end{thebibliography}

\end{document}